\documentclass[11pt, abstract]{scrartcl}
\usepackage{amsmath}
\usepackage{amssymb}
\usepackage{amsthm}
\usepackage{physics}
\usepackage{units}
\usepackage[final]{microtype}
\usepackage{booktabs}
\usepackage{multirow}
\usepackage{xcolor}
\usepackage{tikz}
\usepackage{pgfplots}
\pgfplotsset{compat=newest}
\usetikzlibrary{arrows.meta, external, calc}
\usepgfplotslibrary{colormaps}
\usepackage[backend=biber,hyperref,backref,bibencoding=utf8,defernumbers=true,style=numeric-comp,giveninits=true,maxnames=99,doi=false,isbn=false]{biblatex}
\renewbibmacro{in:}{\ifentrytype{article}{}{\printtext{\bibstring{in}\intitlepunct}}}
\usepackage{csquotes}
\makeatletter
\let\blx@rerun@biber\relax
\makeatother

\usepackage{hyperref}
\usepackage{cleveref}

\theoremstyle{plain}
	\newtheorem{theorem}{Theorem}
	\newtheorem{lemma}{Lemma}
	\newtheorem{proposition}{Proposition}
	\newtheorem{corollary}{Corollary}
\theoremstyle{definition}
	
	\newtheorem{assumption}{Assumption}
	
	\newtheorem{remark}{Remark}

\renewcommand{\vec}[1]{{\boldsymbol{#1}}}
\newcommand{\jump}[1]{[\![{#1}]\!]}

\newcommand{\dist}{{\operatorname{dist}}}
\newcommand{\diam}{{\operatorname{diam}\,}}

\newcommand{\R}{{\mathbb R}}
\newcommand{\N}{{\mathbb N}}
\renewcommand{\P}{{\mathbb P}}

\newcommand{\RecOp}{{I_h}}
    
\newcommand{\Faces}{{\mathcal{F}(\mathcal{T}_h)}}    
\newcommand{\Hdiv}{{\vec{H}(\operatorname{div},\Omega)}}
\newcommand{\Hdivnull}{{\vec{H}_0(\operatorname{div},\Omega)}}

\newcommand{\abssec}[1]{\noindent\small {\bfseries #1\quad}\ignorespaces}
\renewenvironment{abstract}{\abssec{Abstract}}{\par\vspace{1em}}
\newenvironment{keywords}{\abssec{Keywords}}{\par\vspace{1em}}
\newenvironment{MSC}{\abssec{Mathematics Subject Classification (2010)}}{\par\vspace{1em}}

\title{Pressure-robust error estimate\\ of optimal order for the Stokes equations\\ on domains with edges}
\date{\today}
\author{Thomas Apel \and Volker Kempf}

\begin{document}
\maketitle

\begin{abstract}
	The velocity solution of the incompressible Stokes equations is not affected by changes of the right hand side data in form of gradient fields. 
	Most mixed methods do not replicate this property in the discrete formulation due to a relaxation of the divergence constraint which means that they are not pressure-robust. 
	A recent reconstruction approach for classical methods recovers this invariance property for the discrete solution, by mapping discretely divergence-free test functions to exactly divergence-free functions in the sense of $\vec{H}(\operatorname{div})$.
	Moreover, the Stokes solution has locally singular behavior in three-dimensional domains near concave edges, which degrades the convergence rates on quasi-uniform meshes and makes anisotropic mesh grading reasonable in order to regain optimal convergence characteristics. 
	Finite element error estimates of optimal order on meshes of tensor-product type with appropriate anisotropic grading are shown for the pressure-robust modified Crouzeix--Raviart method using the reconstruction approach.
	Numerical examples support the theoretical results.
\end{abstract}
\begin{keywords}	
	anisotropic finite elements, incompressible Stokes equations, divergence-free methods, pressure-robustness, edge singularity
\end{keywords}
\begin{MSC} 
	65N30, 65N15, 65N12
\end{MSC}

\section{Introduction}
When considering polyhedral domains, the solution of the Stokes equations
\begin{subequations}\label{eq:stokes}
\begin{align}
	-\nu \Delta \vec{u} + \nabla p &= \vec{f}, \label{eq:stokes_impulse}\\
	\nabla \cdot \vec{u} &= 0,
\end{align} 
\end{subequations}
shows in general singular behavior near corners and edges. 
On quasi-uniform meshes, this leads to sub-optimal performance of standard numerical methods, which can be remedied by local mesh refinement near the singular sections of the boundary. 
Isotropic refinement can compensate the negative effect of the singular solution, but also leads to over-refinement near edges and thus a waste of computational resources. 
Anisotropic refinement on the other hand can recover the optimal convergence rate \cite{Apel1999,ApelNicaiseSchoberl2001,ApelNicaiseSchoberl2001:2}, while the number of elements $N$ in the mesh still satisfies $N \sim h^{-3}$, where $h$ is the mesh size parameter. 

Unfortunately, many classical mixed methods do not fulfill the discrete inf-sup stability condition independently of the aspect ratio of the triangulation, which may be unbounded in the case of anisotropic grading.
For instance, the standard proof of the inf-sup condition for the Taylor--Hood and Mini-element leads to a constant that depends on the aspect ratio. 
While for the lowest order Taylor--Hood pair a new proof has been found recently that shows inf-sup stability on a class of anisotropic meshes \cite{BarrenecheaWachtel2019}, the Mini-element is reported to become unstable with decreasing minimum angle in the triangulation \cite{AcostaDuran1999}.

However, several inf-sup stable methods are known for anisotropic triangulations in two dimensions, e.g. the Bernardi--Raugel and related elements \cite{ApelNicaise2004} and the stabilized $Q_1/Q_1$, $Q_1/Q_0$ and rotated $\tilde{Q}_1/P_0$ elements for quadrilaterals \cite{Becker1995,BeckerRannacher1995}. 
Additionally, results are available for the $hp$-version finite element method, see e.g. \cite{AinsworthCoggins2000,AinsworthCoggins2002,SchotzauSchwabStenberg1999}.
The Crouzeix--Raviart element \cite{CrouzeixRaviart1973}, which we will focus on in this contribution, is inf-sup stable on simplicial triangulations in two and three dimensions, without any condition on the mesh \cite[Lemma 3.1]{ApelNicaiseSchoberl2001}. 

In addition to its low regularity near concave edges, the velocity solution of the Stokes problem is not affected by changes in form of gradient fields on the right hand side.
This property leads to the notion \emph{velocity-equivalence} of forces, i.e. $\vec{f}_1, \vec{f}_2 \in \vec{L}^2(\Omega)$ are velocity-equivalent, $\vec{f}_1 \simeq \vec{f}_2$, if they lead to the same velocity solution of \eqref{eq:stokes}, see \cite{GaugerLinkeSchroeder2019}. 
That is the case if and only if they differ by a gradient field, see e.g. \cite{AhmedLinkeMerdon2018,JohnLinkeMerdonNeilanRebholz2017,LinkeMerdonNeilan2020}. 
Reproducing this continuous property on the discrete level poses an additional difficulty for discretization schemes, and most classical methods do not overcome it. In fact, error estimates for classical $\vec{H}^1$-conforming methods are of the form, see e.g. \cite{JohnLinkeMerdonNeilanRebholz2017,GiraultRaviart1986},
\begin{equation}
	\norm{\vec{u}-\vec{u}_h}_{1} \leq 2(1+C_F) \inf_{\vec{v}_h\in\vec{X}_h} \norm{\vec{u}-\vec{v}_h}_{1} + \frac{1}{\nu} \inf_{q_h\in Q_h} \norm{p - q_h}_{0},
\end{equation}
where $\norm{\cdot}_k$ is the standard $\vec{H}^k(\Omega)$-Sobolev norm and $C_F$ is the stability constant of the Fortin operator of the mixed method. This type of estimate implies that in settings where the continuous pressure is more difficult to approximate compared to the velocity, the velocity approximation can be highly inaccurate. 

Consider for example a hydrostatic case where the exact velocity is given as $\vec{u}\equiv \vec{0}$ and the continuous pressure is a polynomial of order $k$. Then for classical methods using piecewise polynomials of order less than $k$ for the pressure approximation, in general inaccurate discrete velocity solutions $\vec{u}\not\equiv \vec{0}$ can be observed. In contrast, so called pressure-robust methods, i.e. methods that see the velocity-equivalence of forces, yield the exact velocity solution, even for lowest-order mixed methods with piecewise constant pressure approximation, see \cite[Section 2.5]{GaugerLinkeSchroeder2019}.

Additionally to the naturally pressure-robust class of $\vec{H}(\operatorname{div})$-conforming finite element methods, see e.g. \cite{BoffiBrezziFortin2013,SchroederLehrenfeldLinkeLube2018,Schroeder2019}, a recent approach using a reconstruction operator on the velocity test functions in the linear form showed that most classical mixed methods can be made pressure-robust at the cost of an additional consistency error, see e.g. \cite{LedererLinkeMerdonSchoberl2017,Linke2014,JohnLinkeMerdonNeilanRebholz2017}. 
In \cite{ApelKempfLinkeMerdon2020} the pressure-robust modified Crouzeix--Raviart element was analyzed on anisotropic triangulations, using the assumption of a regular solution, i.e. $(\vec{u},p)\in\vec{H}^2(\Omega) \times H^1(\Omega)$. 
In this article we extend those results to the case of domains with concave edges and low regularity of the exact solutions.

The main contribution of this paper is a pressure-robust estimate for the velocity solution of the modified Crouzeix--Raviart method in low-regularity settings due to non-smooth domains. 
The estimate shows that when appropriate anisotropic mesh grading towards a non-convex edge is used, an optimal convergence rate can be achieved by the pressure robust method.
We provide an estimate on the pressure error for the modified method in the anisotropic setting as well.

The article is structured as follows. \Cref{sec:stokes} introduces the problem and basic notation. 
The type of mesh and the modified Crouzeix--Raviart method is described in \Cref{sec:discretization}, and \Cref{sec:Helmholtz} shows some aspects of the Helmholtz--Hodge decomposition of vector fields which are important to the analysis. 
\Cref{sec:error} contains the a-priori error analysis, \Cref{sec:example} shows the performance of the method with the help of two numerical examples.

\section{Continuous Stokes problem}\label{sec:stokes}
Consider a prismatic domain $\Omega = G \times Z$, where $G$ is a polygonal shape with one concave corner at which the interior angle is $\omega\in (\pi, 2\pi)$, and $Z$ is a bounded interval. 
To facilitate notation we assume that the non-convex corner of $G$ is placed at the origin, i.e. the relevant edge of $\Omega$ is located on the $z$-axis.
On the domain $\Omega$, consider the incompressible Stokes equations \eqref{eq:stokes} with homogeneous Dirichlet boundary condition
\begin{equation*}
	\vec{u} = 0 \quad \text{on } \partial \Omega,
\end{equation*}
where $\nu$ is the kinematic viscosity and vector valued quantities are denoted in bold symbols.
For $\vec{f} \in \vec{L}^2(\Omega)$, the corresponding weak formulation given by
\begin{subequations}\label{eq:StokesContWeak}
\begin{align}
	\nu(\nabla\vec{u},\nabla\vec{v}) - (\nabla \cdot \vec{v},p) &= (\vec{f},\vec{v})	&&\forall \vec{v}\in\vec{X},	\\
	(\nabla \cdot \vec{u},q) &= 0 &&\forall q\in Q,
\end{align}
\end{subequations}
has a unique solution $(\vec{u},p) \in \vec{X}\times Q$, see e.g. \cite[Section I.5.1]{GiraultRaviart1986}, where 
\begin{align*}
	&\vec{X} = \vec{H}^1_0(\Omega) = \{\vec{v} \in \vec{H}^1(\Omega): \vec{v} = 0 \text{ on } \partial \Omega\},	\\
	&Q = L^2_0(\Omega) = \{q\in L^2(\Omega): \begingroup\textstyle\int\endgroup_\Omega q = 0 \},
\end{align*}
and $(\cdot,\cdot)$ denotes the $\vec{L}^2(\Omega)$ scalar product.
With the space of divergence free functions
\begin{equation*}
	\vec{V}^0 = \{\vec{v}\in \vec{X} : (\nabla\cdot \vec{v}, q) = 0 \quad \forall q\in Q\},
\end{equation*} 
we can reformulate the problem, see \cite[Section I.5.1]{GiraultRaviart1986}: find $\vec{u}\in \vec{V}^0$, so that
\begin{equation*}
	\nu (\nabla\vec{u},\nabla\vec{v}) = ( \vec{f},\vec{v} ) \quad \forall \vec{v}\in\vec{V}^0.
\end{equation*}
Additionally to the well known Stokes theory, see e.g. \cite{GiraultRaviart1986}, which states the regularity of the solution in the Hilbert space case as above, Theorem 2.1 in \cite{GaldiSimaderSohr1994} classifies the solution in a more general setting: for $\vec{f}\in\vec{W}^{-1,q}(\Omega)$ and appropriate regularity of the boundary condition we have $(\vec{u},p) \in \vec{W}^{1,q}(\Omega)\times L_0^q(\Omega)$, with $1<q<\infty$.

For the special case of convex prismatic polyhedral domains, we can assume that the solution of problem \eqref{eq:StokesContWeak} satisfies $(\vec{u},p) \in \vec{H}^2(\Omega) \times H^1(\Omega)$, see \cite{Dauge1989}. 
This is in general not the case for non-convex geometries like the ones we are considering, but Theorem 2.1 in \cite{ApelNicaiseSchoberl2001} gives a regularity result for our case in weighted Sobolev spaces. In particular, the derivatives of the solution in the direction parallel to the concave edge have the standard regularity, i.e. $\partial_z \vec{u} \in \vec{H}^1(\Omega)$ and $\partial_z p \in L^2(\Omega)$. The global regularity of $\vec{u}$ is however characterized by $r^\lambda$, where $r$ is the distance to the singular edge and $\lambda$ is the smallest positive solution of 
\begin{equation}\label{eq:lambda}
	\sin(\lambda \omega) = -\lambda \sin(\omega),
\end{equation}
for which $\nicefrac{1}{2} < \lambda < \nicefrac{\pi}{\omega}$ holds, see \cite{Dauge1989}.

\section{Discretization}\label{sec:discretization}
\begin{figure}[t]
	\centering
	\includegraphics[scale=1]{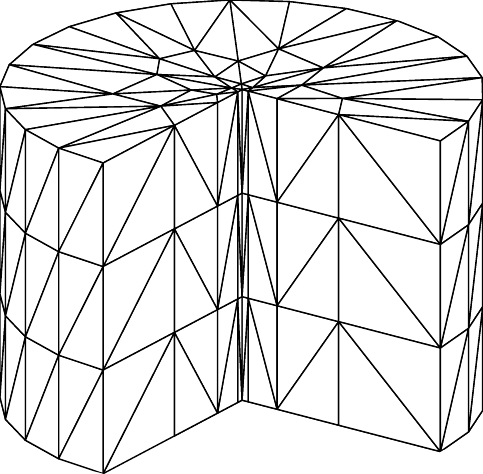}
	\caption{Example mesh with anisotropic grading towards the concave edge}
	\label{fig:mesh_3d}
\end{figure}

\Cref{fig:mesh_3d} shows the type of anisotropically graded tensor-product mesh used for the discretization of the problem, and we briefly describe the process of mesh generation in the following paragraph. 
This type of mesh was introduced in \cite{ApelDobrowolski1992} for the treatment of edge singularities that occur in the Poisson problem, and was used in subsequent works also for the Stokes and Maxwell equations \cite{ApelNicaiseSchoberl2001,ApelNicaiseSchoberl2001:2,FarhloulNicaisePaquet2001,Nicaise2001}.  

Let $\mathcal{D}_h$ be a conforming, shape regular triangulation of the two-dimensional domain $G$, which has a mesh size parameter $h= \max_{D\in\mathcal{D}_h} h_D$, where $h_D =\diam(D)$. 
This mesh is graded towards the non-convex corner, so that the size of every element satisfies
\begin{equation*}
	h_D \sim 	\begin{cases}
					h^{\nicefrac{1}{\mu}}, &\text{if } r_D = 0, \\
					h r_D^{1-\mu}, &\text{if } 0<r_D<R, \\
					h,	&\text{else},
				\end{cases}
\end{equation*}
where $r_D=\inf_{\vec{x}\in D}\{\dist(\vec{x}, \vec{0})\}$ is the distance of an element $D\in \mathcal{D}_h$ to the concave corner, $\mu \in (0,1]$ is a grading parameter and $R>0$ is the radius of the refinement zone. 
The graded two dimensional mesh is extended into the $z$-direction with uniform mesh size $h_3 \sim h$. The resulting prisms are subdivided into tetrahedra, which form the simplicial mesh $\mathcal{T}_h$. 
With $r_T$ being the distance of an element $T\in \mathcal{T}_h$ to the $z$-axis and $h_{1,T}$, $h_{2,T}$, and $h_{3,T}$ the length of the projection on the $x$-, $y$-, and $z$-axis, respectively, the procedure yields a mesh where
\begin{equation*}
	h_{3,T} \sim h,\quad h_{1,T} \sim h_{2,T} \sim 	
								\begin{cases}
									h^{\nicefrac{1}{\mu}}, &\text{if } r_T = 0, \\
									h r_T^{1-\mu}, &\text{if } 0<r_T<R, \\
									h,	&\text{else},
								\end{cases}
\end{equation*}
and the number of elements satisfies $N \sim h^{-3}$. By $\Faces$ we denote the set of facets of the mesh $\mathcal{T}_h$.
\begin{remark}\label{rem:MAC}
	By construction, this type of tensor-product mesh satisfies a maximum angle condition, i.e. all angles between edges and faces of the triangulation are bounded by a constant $\bar{\psi} < \pi$. 
	The subsequent analysis depends on this regularity assumption on the tetrahedra, which means that meshes like the ones used in \cite{LiNicaise2018}, where the maximum angle condition is violated, can not easily be included in the theory.
\end{remark}

Our discretization is nonconforming, thus we need tools to handle potential discontinuities at the interfaces. Let $\jump{\vec{v}}_F$ be the jump of a function $\vec{v}$ over a facet $F$, which is defined for an interior facet belonging to two elements $T_1$ and $T_2$ by
\begin{align*}
	\jump{\vec{v}}_F(\vec{x}) &= \vec{v}|_{T_1}(\vec{x}) - \vec{v}|_{T_2}(\vec{x}),
\end{align*}
see e.g. \cite[Section 1.2.3]{DiPietroErn2012}. On boundary facets we set $\jump{\vec{v}}_F = \vec{v}$. For the velocity approximation we use the Crouzeix--Raviart finite element function space
\begin{equation*}
	\vec{X}_h = \{\vec{v}_h\in \vec{L}^2(\Omega) : \vec{v}_{h}|_T \in \vec{P}_1(T)\quad \forall T\in\mathcal{T}_h,\ \jump{\vec{v}_h}_F(\vec{x}_F) = 0\quad \forall F \in \Faces\},
\end{equation*}
which was introduced in \cite{CrouzeixRaviart1973}, and where $\vec{x}_F$ is the barycenter of a facet $F$. The pressure is approximated in the space of piecewise constants
\begin{equation*}
	Q_h = \{q_h\in Q : q_{h}|_T \in P_0(T)\quad \forall T \in \mathcal{T}_h\},
\end{equation*}
where $P_k(T)$ denotes the space of all polynomials with maximal degree $k$ on the element $T$. 
We also need the broken gradient $\nabla_h:\vec{X}\oplus\vec{X}_h\rightarrow L^2(\Omega)^{d\times d}$ and the broken divergence $\nabla_h\cdot(\cdot) : \vec{X}\oplus\vec{X}_h\rightarrow L^2(\Omega)$, which define the derivatives elementwise for all $T\in\mathcal{T}_h$ by
\begin{align*}
	(\nabla_h\vec{v}_h)|_T &= \nabla(\vec{v}_h|_T),	\\
	(\nabla_h\cdot \vec{v}_h)|_T &= \nabla\cdot(\vec{v}_h|_T),
\end{align*}
and which are on $\vec{X}$ equivalent to the standard operators, see e.g. \cite[Sections 1.2.5, 1.2.6]{DiPietroErn2012}. 
The discrete gradient norm for the space $\vec{X}\oplus\vec{X}_h$ is defined by
\begin{equation*}
	\norm{\vec{v}_h}_{1,h} = \left(\int_\Omega \nabla_h\vec{v}_h : \nabla_h\vec{v}_h\right)^{\nicefrac{1}{2}} = \norm{\nabla_h \vec{v}_h}_0.
\end{equation*}

For the next part we need the function spaces
\begin{align*}
	\Hdiv &= \{ \vec{v} \in \vec{L}^2(\Omega) : \nabla\cdot \vec{v} \in L^2(\Omega) \}, \\
	\Hdivnull &= \{\vec{v} \in \Hdiv : \vec{v}\cdot \vec{n} = 0 \text{ on } \partial \Omega\},
\end{align*}
where $\vec{n}$ denotes the unit outward normal vector on $\partial \Omega$.
Our discretization uses a reconstruction operator on the velocity test functions in the linear form, and in order to yield a pressure-robust method the operator needs to satisfy some properties, see e.g. \cite{Linke2014,BrenneckeLinkeMerdonSchoberl2015,LinkeMerdonNeilan2020}, which we summarize in the following assumption.
\begin{assumption} \label{as:RecOp}
	Assume there is a reconstruction operator $\RecOp:\vec{X}\oplus\vec{X}_h \rightarrow \vec{Y}_h$, where $\vec{Y}_h \subset \Hdivnull$, so that for all $\vec{v}_h \in \vec{X}_h$
	\begin{align}
		\nabla \cdot (\RecOp \vec{v}_h) &= \nabla_h \cdot \vec{v}_h, \nonumber\\
		\norm{\vec{v}_h - \RecOp \vec{v}_h}_0 &\leq c h \norm{\nabla_h \vec{v}_h}_0. \label{eq:interpolation_error}
	\end{align}
\end{assumption}
 
When using the approximation spaces $\vec{X}_h$ and $Q_h$, the lowest-order Raviart--Thomas and Brezzi--Douglas--Marini interpolation operators satisfy this assumption, see \cite{BrenneckeLinkeMerdonSchoberl2015,Linke2014} for the isotropic and \cite{ApelKempfLinkeMerdon2020} for the anisotropic case. 
For our intended use of the method in an anisotropic setting, the constant in estimate \eqref{eq:interpolation_error} must be independent of the aspect ratio of the mesh. 
Under the mild assumption of the maximum angle condition which is satisfied for the type of mesh described above, see \Cref{rem:MAC}, this is the case for both Raviart--Thomas interpolation, see \cite{AcostaApelDuranLombardi2011}, and Brezzi--Douglas--Marini interpolation, see \cite{ApelKempf2020}.

Using the discrete bilinear forms
\begin{alignat*}{2}
	&a_h:\vec{X}_h\times\vec{X}_h \rightarrow \R,&&\quad a_h(\vec{u}_h,\vec{v}_h) = \nu \int_\Omega \nabla_h\vec{u}_h : \nabla_h\vec{v}_h,	\\
	&b_h:\vec{X}_h\times Q_h \rightarrow \R,&&\quad b_h(\vec{v}_h, q_h) = -\int_\Omega q_h \nabla_h\cdot\vec{v}_h,
\end{alignat*}
we get the discrete weak formulation
\begin{subequations}\label{eq:stokesweak}
\begin{align}
		a_h(\vec{u}_h,\vec{v}_h) + b_h(\vec{v}_h,p_h) &= ( \vec{f}, \RecOp \vec{v}_h )  &&\forall \vec{v}_h\in \vec{X}_h, \\
		b_h(\vec{u}_h,q_h) &= 0 &&\forall q_h\in Q_h,
\end{align}
\end{subequations}
where $\vec{f}\in\vec{L}^2(\Omega)$ and $\RecOp$ must satisfy \Cref{as:RecOp}.
As in the continuous case, using the space of discretely divergence constrained functions
\begin{equation*}
	\vec{V}^0_h = \left\{\vec{v}_h\in \vec{X}_h : b_h(\vec{v}_h, q_h) = 0 \ \forall q_h\in Q_h\right\} = \{\vec{v}_h \in \vec{X}_h : \nabla_h \cdot \vec{v}_h = 0 \},
\end{equation*}
we write can rewrite the problem, see \cite{LinkeMerdonWollner2017,BrenneckeLinkeMerdonSchoberl2015,GiraultRaviart1986}. Thus $\vec{u}_h\in \vec{V}^0_h$ is uniquely defined by
\begin{equation}\label{eq:stokeselliptic}
	a_h(\vec{u}_h,\vec{v}_h) = ( \vec{f}, \RecOp\vec{v}_h ) \quad \forall \vec{v}_h\in\vec{V}^0_h.	
\end{equation}

To conclude this section, we state the well-known discrete inf-sup stability for the Crouzeix--Raviart element, see e.g. \cite[Lemma 3.1]{ApelNicaiseSchoberl2001}.
\begin{lemma}
	The pair of function spaces $\vec{X}_h \times Q_h$ satisfies the discrete inf-sup condition
	\begin{equation}\label{eq:infsup}
		\inf_{q_h\in Q_h\setminus \{0\}} \sup_{\vec{v}_h \in \vec{X}_h\setminus\{0\}} \frac{b_h(\vec{v}_h, q_h)}{\norm{q_h}_0 \norm{\vec{v}_h}_{1,h}} \geq \tilde{\beta} > 0,
	\end{equation}
	where the discrete inf-sup constant $\tilde{\beta}$ does not depend on the mesh size parameter $h$ or the regularity of the mesh.
\end{lemma}

\section{Helmholtz--Hodge decomposition}\label{sec:Helmholtz}
This section introduces some aspects of the Helmholtz--Hodge decomposition of vector fields, which is needed for overall context and explanation.
The main idea of this section is from \cite[Section 3]{LinkeMerdonNeilan2020}. 

Every vector field $\vec{g} \in \vec{L}^2(\Omega)$ can be uniquely decomposed into $\vec{g} = \P(\vec{g}) + \nabla \phi$, where $\phi \in H^1(\Omega)/\R$ and 
\begin{equation*}
	\P(\vec{g}) \in \vec{L}_\sigma^2(\Omega) = \{\vec{v}\in\vec{L}^2(\Omega) : (\nabla q, \vec{v}) = 0\quad \forall q\in H^1(\Omega)\}.
\end{equation*}
The function $\P(\vec{g})$ is called Helmholtz--Hodge projection of $\vec{g}$, see e.g. \cite[Corollary I.3.4]{GiraultRaviart1986}. The operator $\P(\cdot):\vec{L}^2(\Omega) \rightarrow \vec{L}_\sigma^2(\Omega)$ is an $\vec{L}^2$-orthogonal projection, i.e. 
\begin{equation*}
	(\P(\vec{g}),\vec{v}) = (\vec{g}, \vec{v}) \quad \forall \vec{v}\in \vec{L}_\sigma^2(\Omega).
\end{equation*}

We can extend the domain of the Helmholtz--Hodge projection operator from $\vec{L}^2(\Omega)$ to $\vec{H}^{-1}(\Omega)$ with range in $(\vec{V}^0)'$, the dual space of $\vec{V}^0$, by defining the projection for every $\vec{g}\in \vec{H}^{-1}(\Omega)$ as the restriction to $\vec{V}^0$, i.e.
\begin{equation*}
	\langle \P(\vec{g}), \vec{v} \rangle = \langle \vec{g}, \vec{v} \rangle \quad \forall \vec{v}\in \vec{V}^0.
\end{equation*}
Note that it holds $\vec{V}^0 \subset \vec{L}^2_\sigma(\Omega)$. For a more detailed and technical introduction of this extension we refer to \cite[Section 2]{Monniaux2006}. A functional $\vec{g}^* \in \vec{H}^{-1}(\Omega)$ with $\vec{L}^2$-representative $\vec{g}$, has the Helmholtz--Hodge projection $\P(\vec{g}^*)\in (\vec{V}^0)'$ with representative $\P(\vec{g}) \in \vec{L}^2(\Omega)$, since by the previous definitions and the Riesz representation theorem it holds for all $\vec{v}\in \vec{V}^0$
\begin{equation*}
	\langle \P(\vec{g}^*) , \vec{v} \rangle = \langle \vec{g}^*, \vec{v} \rangle = (\vec{g}, \vec{v}) = (\P(\vec{g}), \vec{v}).
\end{equation*}

Defining $-\Delta : \vec{H}_0^1(\Omega) \rightarrow \vec{H}^{-1}(\Omega)$ by
\begin{equation*}
	\langle -\Delta \vec{v} , \vec{\psi} \rangle = ( \nabla \vec{v}, \nabla \vec{\psi}) \quad \forall \vec{\psi} \in \vec{H}_0^1(\Omega),
\end{equation*}
according to Lemma 3.1 in \cite{LinkeMerdonNeilan2020} the equality
\begin{equation*}
	\P(-\Delta\vec{u}) = \nu^{-1} \P(\vec{f})
\end{equation*}
holds for the weak Stokes velocity solution $\vec{u}$ with data $\vec{f}$. This means that although in general $-\Delta\vec{u} \in \vec{H}^{-1}(\Omega)$ even for data $\vec{f}\in\vec{L}^2(\Omega)$, it holds $\P(-\Delta \vec{u}) \in \vec{L}^2(\Omega)$ and
\begin{equation}\label{eq:projectoridentity}
	\nu\norm{\P(-\Delta\vec{u})}_0 = \norm{\P(\vec{f})}_0\leq\norm{\vec{f}}_0.
\end{equation}

\begin{lemma}\label{lem:solution_change}
	If $(\vec{u}, p)$ is the solution of \eqref{eq:StokesContWeak} with data $\vec{f} = \P(\vec{f}) + \nabla \phi$, then $(\vec{u}, \nu^{-1}(p-\phi))$ is the solution of the Stokes equations with unit viscosity and data $\nu^{-1} \P(\vec{f})$.
\end{lemma}
\begin{proof}
	The Stokes equations satisfy a fundamental invariance property, i.e. adding a gradient field to the data only changes the pressure solution, see \cite{JohnLinkeMerdonNeilanRebholz2017}. 
	Thus $(\vec{u}, p - \phi)$ is the solution with data function $\P(\vec{f}) = \vec{f}-\nabla \phi$. 
	Dividing the momentum equation by $\nu$, we get the statement of the lemma.
\end{proof}

\section{A-priori error estimates}\label{sec:error}
For an estimate on the finite element error, the consistency error of the method has to be estimated. For self-containedness we restate \cite[Lemma 3.3]{ApelNicaiseSchoberl2001} which estimates the consistency error for the standard method in the case $\nu=1$. 
\begin{lemma}\label{lem:ANS_consistency}
	Let $(\vec{u}, p)$ be the solution of the Stokes problem with $\nu=1$ and data $\vec{f} \in \vec{L}^2(\Omega)$. Then if the mesh is refined according to $\mu < \lambda$, with $\lambda$ from \eqref{eq:lambda}, the estimate 
	\begin{equation*}
		\abs{(\nabla_h \vec{u}, \nabla_h \vec{v}_h) + b_h(\vec{v}_h,p) - (\vec{f},\vec{v}_h)} \leq c h \norm{\vec{v}_h}_{1,h} \norm{\vec{f}}_0
	\end{equation*}
	holds for all $\vec{v}_h \in \vec{X}_h$. For $\vec{v}_h \in \vec{V}_h^0$ we have the estimate
	\begin{equation*}
		\abs{(\nabla_h \vec{u}, \nabla_h \vec{v}_h) - (\vec{f},\vec{v}_h)} \leq c h \norm{\vec{v}_h}_{1,h} \norm{\vec{f}}_0.
	\end{equation*}
\end{lemma}
\begin{proof}
	The first inequality is the statement from \cite[Lemma 3.3]{ApelNicaiseSchoberl2001}, and the second holds since for $\vec{v}_h \in \vec{V}_h^0$ we have $\nabla_h \cdot\vec{v}_h = 0$, and thus get
	\begin{equation*}
		b_h(\vec{v}_h, p) = -(\nabla_h \cdot \vec{v}_h, p) = 0. \qedhere
	\end{equation*}
\end{proof}

We can now state the error estimate of the velocity solution of our method. 
It shows that for appropriately refined meshes the method has an optimal order of convergence and is pressure-robust, i.e. the estimate does not depend on the viscosity or the pressure approximability.
\begin{theorem}\label{th:velocity_error}
	Let $(\vec{u}, p)$ be the solution of \eqref{eq:StokesContWeak}, $(\vec{u}_h, p_h)$ the solution of \eqref{eq:stokesweak}, and let the mesh be refined according to $\mu < \lambda$, with $\lambda$ from \eqref{eq:lambda}. In addition, let the reconstruction operator satisfy \Cref{as:RecOp}. Then we have the estimate
	\begin{equation}\label{eq:velocity_error}
		\norm{\vec{u}-\vec{u}_h}_{1,h} \leq 2 \inf_{\vec{v}_h \in \vec{V}_h^0} \norm{\vec{u}-\vec{v}_h}_{1,h} + c h \norm{\P(-\Delta \vec{u})}_0.
	\end{equation}
\end{theorem}
\begin{proof}
	Let $\vec{w}_h = \vec{u}_h - \vec{v}_h \in \vec{V}_h^0$, for arbitrary $\vec{v}_h \in \vec{V}_h^0$. Then using the triangle inequality we estimate
	\begin{equation}\label{eq:proof1}
		\norm{\vec{u}-\vec{u}_h}_{1,h} = \norm{\vec{u}-\vec{v}_h-\vec{w}_h}_{1,h} \leq \norm{\vec{u}-\vec{v}_h}_{1,h} + \norm{\vec{w}_h}_{1,h}.
	\end{equation}
	Using \eqref{eq:stokeselliptic} we get 
	\begin{align*}
		\nu \norm{\vec{w}_h}_{1,h}^2 &= a_h(\vec{w}_h, \vec{w}_h) = a_h(\vec{u}_h - \vec{v}_h, \vec{w}_h) \\
		&= a_h(\vec{u} - \vec{v}_h, \vec{w}_h) - a_h(\vec{u}, \vec{w}_h)  + a_h(\vec{u}_h, \vec{w}_h)\\
		&= a_h(\vec{u} - \vec{v}_h, \vec{w}_h) - a_h(\vec{u}, \vec{w}_h) + (\vec{f}, \RecOp\vec{w}_h) \\
		&\leq \nu \norm{\vec{u}-\vec{v}_h}_{1,h} \norm{\vec{w}_h}_{1,h} + \abs{a_h(\vec{u}, \vec{w}_h) - (\vec{f}, \RecOp\vec{w}_h) },
	\end{align*}
	which combined with \eqref{eq:proof1} gives
	\begin{equation}\label{eq:proof2}
		\norm{\vec{u}-\vec{u}_h}_{1,h} \leq 2 \norm{\vec{u}-\vec{v}_h}_{1,h} + \nu^{-1} \frac{\abs{a_h(\vec{u}, \vec{w}_h) - (\vec{f}, \RecOp\vec{w}_h)}}{\norm{\vec{w}_h}_{1,h}}.
	\end{equation}
	 Denote the Helmholtz--Hodge decomposition of the data by $\vec{f} = \P(\vec{f}) + \nabla \phi$ and note that $\nabla\cdot\RecOp \vec{w}_h = 0$ due to \Cref{as:RecOp} and $\vec{w}_h \in \vec{V}_h^0$.
	Using $(\nabla \phi, \RecOp\vec{w}_h) = 0$, we get 
	\begin{align}
		\abs{a_h(\vec{u}, \vec{w}_h) - (\vec{f}, \RecOp\vec{w}_h)} &= \abs{a_h(\vec{u}, \vec{w}_h) - (\P(\vec{f}), \RecOp \vec{w}_h)} \nonumber\\
		&= |a_h(\vec{u}, \vec{w}_h) - (\P(\vec{f}), \vec{w}_h) + (\P(\vec{f}), \vec{w}_h - \RecOp \vec{w}_h) | \nonumber\\
		&\leq \abs{a_h(\vec{u}, \vec{w}_h) - (\P(\vec{f}), \vec{w}_h)} + \abs{(\P(\vec{f}), \vec{w}_h - \RecOp \vec{w}_h)}. \label{eq:proof3}
	\end{align}
	Due to \Cref{lem:solution_change} we can estimate the first term of \eqref{eq:proof3} using \Cref{lem:ANS_consistency} and get
	\begin{align}\label{eq:proof4}
		\abs{a_h(\vec{u}, \vec{w}_h) - (\P(\vec{f}), \vec{w}_h)}  &= \nu \abs{ (\nabla_h \vec{u}, \nabla_h \vec{w}_h) - (\nu^{-1} \P(\vec{f}), \vec{w}_h) } \nonumber\\
		&\leq c h \nu \norm{\vec{w}_h}_{1,h} \norm{\nu^{-1} \P(\vec{f})}_0 = c h \norm{\vec{w}_h}_{1,h} \norm{\P(\vec{f})}_0.
	\end{align}
	Using the Cauchy-Schwarz inequality and the interpolation error estimate for the operator $\RecOp$ from \Cref{as:RecOp} we estimate for the second term of \eqref{eq:proof3}
	\begin{equation}\label{eq:proof5}
		 \abs{(\P(\vec{f}), \vec{w}_h - \RecOp \vec{w}_h)} \leq c h \norm{\vec{w}_h}_{1,h} \norm{\P(\vec{f})}_0.
	\end{equation}
	Combining estimates \eqref{eq:proof4}, \eqref{eq:proof5} with \eqref{eq:proof3}, inserting the result in \eqref{eq:proof2}, using \eqref{eq:projectoridentity} and by seeing that $\vec{v}_h$ was chosen arbitrarily, we get the final estimate
	\begin{equation*}
		\norm{\vec{u}-\vec{u}_h}_{1,h} \leq 2 \inf_{\vec{v}_h \in \vec{V}_h^0} \norm{\vec{u}-\vec{v}_h}_{1,h} + c h \norm{\P(-\Delta \vec{u})}_0. \qedhere
	\end{equation*}
\end{proof}
The term for the approximation error can be easily bounded using known results:
\begin{corollary}\label{cor:velocity_order}
	With the assumptions from \Cref{th:velocity_error} the estimate
	\begin{equation*}
		\norm{\vec{u}-\vec{u}_h}_{1,h} \leq c h \norm{\P(-\Delta \vec{u})}_0.
	\end{equation*}
	holds.
\end{corollary}
\begin{proof}
	Using Lemma 8 from \cite{ApelKempfLinkeMerdon2020} we get
	\begin{equation*}
		\inf_{\vec{v}_h \in \vec{V}_h^0} \norm{\vec{u}-\vec{v}_h}_{1,h} \leq 2\inf_{\vec{v}_h \in \vec{X}_h} \norm{\vec{u}-\vec{v}_h}_{1,h}.
	\end{equation*}
	By \Cref{lem:solution_change} $\vec{u}$ is also the velocity solution of the Stokes problem with unit viscosity and right hand side data $\nu^{-1}\P(\vec{f})$, and thus using Lemma 3.2 from \cite{ApelNicaiseSchoberl2001} and \eqref{eq:projectoridentity} we get 
	\begin{equation}\label{eq:approx_error_velocity}
		\inf_{\vec{v}_h \in \vec{V}_h^0} \norm{\vec{u}-\vec{v}_h}_{1,h} \leq 2\inf_{\vec{v}_h \in \vec{X}_h} \norm{\vec{u}-\vec{v}_h}_{1,h} \leq c h \norm{\P(-\Delta \vec{u})}_0,
	\end{equation}
	which combined with \eqref{eq:velocity_error} proves the statement.
\end{proof}

\begin{remark}\label{rem:data_exact_solution}
	Considering \Cref{lem:solution_change}, the relationship between the data $\vec{f}$ and the velocity solution $\vec{u}$ with regard to the viscosity parameter $\nu$ can be looked at from different points of view. 
	On the one hand, in \Cref{th:velocity_error} we establish a velocity error estimate in terms of the divergence-free part of the Laplacian of the exact velocity $\P(-\Delta\vec{u})$. 
	In this form, the estimate is pressure-robust, i.e. it does not depend on the irrotational part of the data, and it does not have an apparent dependence on the viscosity.
	If on the other hand by using \eqref{eq:projectoridentity} we would put the estimate in terms of the Helmholtz-Hodge projection $\P(\vec{f})$ of the data, it would still be pressure-robust, but we would see a dependence on $\nu^{-1}$.
	
	The difference is of interest for numerical examples and the information we want to extract from them. Consider e.g. the examples from \cite[Section 5]{Linke2014}. Here the exact velocity and pressure solutions are fixed, and the data function $\vec{f}$ changes when the viscosity parameter is adjusted due to the factor $\nu$ in front of the Laplacian. 
	This can nicely show the effect of pressure-robustness, since non-pressure-robust methods show a viscosity induced locking effect, i.e. the velocity error scales with $\nu^{-1}$, while pressure-robust methods do not, as the discrete velocity solution is the same for all values of $\nu$.
	If however the data function $\vec{f}$ is fixed, we also see a dependence on the viscosity in the error for pressure-robust methods, since the velocity solution now scales with $\nu^{-1}$. 
	When altering the viscosity parameter while using fixed data, pressure-robustness can still be observed by changing the irrotational part of $\vec{f}$, i.e. adding a gradient field, which has no effect on the numerical velocity solution of pressure-robust methods. 
\end{remark}

For the pressure error we get the following estimate.
\begin{proposition}
	With the assumptions of \Cref{th:velocity_error} we have the estimate
	\begin{equation}\label{eq:pressure_error}
		\norm{p-p_h}_{0} \leq  \inf_{q_h \in Q_h} \norm{p-q_h}_{0} + \frac{2 \nu}{\tilde{\beta}} \inf_{\vec{v}_h \in \vec{V}_h^0} \norm{\vec{u}-\vec{v}_h}_{1,h} + \frac{c h}{\tilde{\beta}}\norm{\vec{f}}_0.
	\end{equation}
\end{proposition}
\begin{proof}
	Let $\pi_h:L^2_0(\Omega) \rightarrow Q_h$ be the $L^2$-orthogonal projection onto the discrete pressure space. We start with a triangle inequality, which gives
	\begin{equation*}
		\norm{p-p_h}_{0} \leq \norm{p-\pi_h p}_{0} + \norm{\pi_h p - p_h}_{0},
	\end{equation*}
	where we see that for the first term it holds $\norm{p-\pi_h p}_{0} = \inf_{q_h \in Q_h} \norm{p-q_h}_{0}$. Because of $\pi_h p - p_h \in Q_h$, we can use the inf-sup condition \eqref{eq:infsup} and estimate
	\begin{equation*}
		\norm{\pi_h p - p_h}_{0} \leq \frac{1}{\tilde{\beta}} \sup_{\vec{v}_h\in \vec{X}_h} \frac{b_h(\vec{v}_h,\pi_h p - p_h)}{\norm{\vec{v}_h}_{1,h}} = \frac{1}{\tilde{\beta}} \sup_{\vec{v}_h\in \vec{X}_h} \frac{b_h(\vec{v}_h,\pi_h p - p) + b_h(\vec{v}_h, p - p_h)}{\norm{\vec{v}_h}_{1,h}},
	\end{equation*}
	where $b_h(\vec{v}_h,\pi_h p - p) = 0$, as $\nabla_h \cdot \vec{v}_h \in Q_h$ and $\pi_h p - p \in Q_h^{\perp_{L^2}}$. Since $p_h$ is the discrete pressure solution of \eqref{eq:stokesweak} we can further calculate
	\begin{align*}
		\norm{\pi_h p - p_h}_{0} &\leq \frac{1}{\tilde{\beta}} \sup_{\vec{v}_h\in \vec{X}_h} \frac{b_h(\vec{v}_h,p)+a_h(\vec{u}_h,\vec{v}_h) - (\vec{f}, \RecOp \vec{v}_h)}{\norm{\vec{v}_h}_{1,h}} \\
		&= \frac{\nu}{\tilde{\beta}} \left(\sup_{\vec{v}_h\in \vec{X}_h} \frac{(\nabla_h \vec{u},\nabla_h \vec{v}_h) + b_h(\vec{v}_h,\nu^{-1}p) - (\nu^{-1}\vec{f}, \vec{v}_h) }{\norm{\vec{v}_h}_{1,h}}\right. \\
		&\left.\hspace{2cm}+ \sup_{\vec{v}_h\in \vec{X}_h} \frac{(\nabla_h (\vec{u}_h-\vec{u}),\nabla_h \vec{v}_h) + (\nu^{-1}\vec{f},\vec{v}_h - \RecOp \vec{v}_h)}{\norm{\vec{v}_h}_{1,h}} \right)\\
		&\leq \frac{\nu}{\tilde{\beta}} \left( \norm{\vec{u}-\vec{u}_h}_{1,h} + c h \norm{\nu^{-1} \vec{f}}_0 \right),
	\end{align*}
	where in the last step we used \Cref{lem:ANS_consistency}, the Cauchy-Schwarz inequality and the interpolation error estimate \eqref{eq:interpolation_error} for the reconstruction operator. Now combining the estimates, using \Cref{th:velocity_error} and \eqref{eq:projectoridentity} yields the desired inequality.
\end{proof}

\begin{corollary}
	With the assumptions of \Cref{th:velocity_error} we have the estimate
	\begin{equation*}
		\norm{p-p_h}_{0} \leq  c h \left(1+\frac{1}{\tilde{\beta}}\right) \norm{\vec{f}}_0.
	\end{equation*}
\end{corollary}
\begin{proof}
	The estimate is obtained from \eqref{eq:pressure_error} by using \cite[Lemma 3.2]{ApelNicaiseSchoberl2001} for the first term and \eqref{eq:approx_error_velocity} in combination with \eqref{eq:projectoridentity} for the second term.
\end{proof}

\begin{remark}
	We consider only the three dimensional case, since the focus of this paper is on anisotropic elements.
	The main results are nevertheless valid for the corresponding two-dimensional problem in a domain with a re-entrant corner, as long as adequate local mesh grading near the corner, as described in the first part of \Cref{sec:discretization}, is applied.
	
	The proofs for the intermediate results from \cite{ApelNicaiseSchoberl2001:2} can be adapted to fit the two-dimensional setting. 
	With them, the consistency error for the standard method, see \Cref{lem:ANS_consistency}, can be proved analogously to the first part of the proof of Lemma 3.2 in \cite{ApelNicaiseSchoberl2001}, without the additional difficulty for the third component. From there, our proofs in this section apply analogously. 
\end{remark}

\section{Numerical Examples}\label{sec:example}
With the following two examples we show the performance of the pressure-robust modified Crouzeix--Raviart method with Raviart--Thomas (CR-RT) and Brezzi--Douglas--Marini (CR-BDM) reconstruction compared to the standard Crouzeix--Raviart (CR) method on anisotropically graded meshes. 
Considering \Cref{rem:data_exact_solution}, we first choose the approach of fixing an exact solution, where the data changes when altering the viscosity. 
However, for our specified exact solution we get $\vec{f} \notin \vec{L}^2(\Omega)$ for $\nu \neq 1$, which does not comply with the assumptions of our theory.
Thus for the second example we use the other approach, where the divergence free part $\P(\vec{f})$ of the data is fixed and only the irrotational part of $\vec{f}$ is changed in order to show pressure-robustness.

\subsection{Example with fixed exact solution} 
Consider the inhomogeneous Stokes problem, i.e. problem \eqref{eq:stokes} with boundary condition $\vec{u} = \vec{g}$ on $\partial \Omega$, on the domain
\begin{equation*}
	\Omega = \{(r\cos(\varphi), r\sin(\varphi), z) \in \R^3: 0<r<1, 0<\varphi<\omega, 0<z<1\}, 
\end{equation*}
where $\omega = \frac{3\pi}{2}$. The results below show that the change to inhomogeneous boundary conditions does not impact the performance of the numerical method. We use the exact velocity and exact pressure solutions defined by
\begin{align*}
	\vec{u} &= 	\begin{pmatrix}
					z r^\lambda [-\lambda\sin(\varphi)\cos(\lambda(\omega-\varphi)+\varphi)+\lambda\sin(\omega-\varphi)\cos(\lambda\varphi-\varphi)+\sin(\lambda(\omega-\varphi))] \\
					z r^\lambda [\sin(\lambda\varphi)-\lambda\sin(\varphi)\sin(\lambda(\omega-\varphi)+\varphi)-\lambda\sin(\omega-\varphi)\sin(\lambda\varphi-\varphi)] \\
					r^{\nicefrac{2}{3}}\sin\left(\frac{2}{3}\varphi\right)
				\end{pmatrix}, \\
	p &= z r^{\lambda-1} \Phi(\varphi),
\end{align*}
where we use 
\begin{equation}\label{eq:Phi}
	\Phi(\varphi) = 2 \lambda [\sin(\omega+(\lambda-1)\varphi) - \sin(\lambda\omega-(\lambda-1)\varphi)].
\end{equation}
From \eqref{eq:lambda} we get $\lambda \approx 0.54448$. The velocity solution and the singular nature of the exact pressure along the edge at $r=0$ are illustrated in \Cref{fig:exact_sol}.
\begin{figure}[t]
	\centering
	\includegraphics[scale=1]{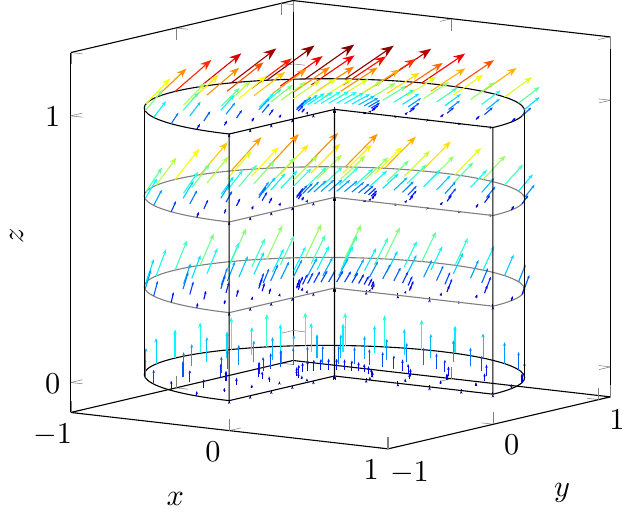}
	\hfill
	\includegraphics[scale=1]{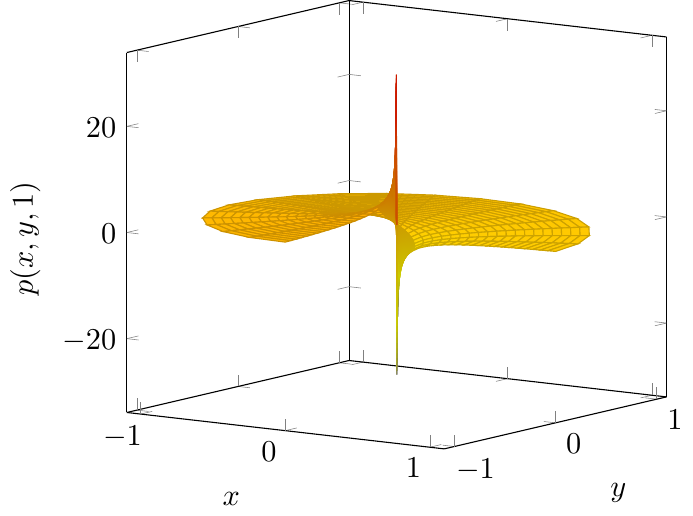}
	\caption{Plot of exact velocity $\vec{u}(x,y,z)$ and exact pressure $p(x,y,1)$}
	\label{fig:exact_sol}
\end{figure}
The data function $\vec{f}$ for the numerical calculations is obtained by evaluating \eqref{eq:stokes_impulse}, from which we get
\begin{equation}\label{eq:datafunction}
	\vec{f} = 	\begin{pmatrix}
					2 \lambda (\lambda -1) (\nu - 1) z r^{\lambda - 2} \left[ \sin(\lambda \omega - (\lambda - 2) \varphi) - \sin(\omega + (\lambda - 2) \varphi) \right] \\
					2 \lambda (\lambda -1) (1 - \nu) z r^{\lambda - 2} \left[ \cos(\lambda \omega - (\lambda - 2) \varphi) + \cos(\omega + (\lambda - 2) \varphi) \right] \\
					r^{\lambda-1} \Phi(\varphi)
				\end{pmatrix},
\end{equation}
where $f_1 = f_2 = 0$ for $\nu = 1$.

In \cite{ApelKempfLinkeMerdon2020} this example was used to show that the modified Crouzeix--Raviart method can be used for anisotropic meshes. 
However, no theoretical foundation for the numerical results was given, since due to the low regularity of the solution in this example, i.e. $(\vec{u},p) \notin \vec{H}^2(\Omega) \times H^1(\Omega)$, $\Delta\vec{u}\notin \vec{L}^2(\Omega)$, the results from \cite{ApelKempfLinkeMerdon2020} are not directly applicable. 
This gap in the theory is closed by this contribution, at least for the case $\nu=1$ where $\vec{f}\in\vec{L}^2(\Omega)$.

As mentioned in \Cref{sec:Helmholtz} we know that $\P(-\Delta\vec{u})\in \vec{L}^2(\Omega)$, since for $\nu = 1$ by \cite[Theorem 2.1]{ApelNicaiseSchoberl2001} it holds $\partial_z p\in L^2(\Omega)$ and thus the data function $\vec{f} = (0,0,\partial_z p)$ is in $\vec{L}^2(\Omega)$. 
As $\vec{u}$ is fixed, this does not change for other values of $\nu$, even though $\vec{f}$ is no longer in $\vec{L}^2(\Omega)$ for $\nu\neq 1$.
\begin{table}[t]
	\centering
	\caption{Errors and experimental convergence orders of the standard and modified Crouzeix--Raviart methods on uniform and graded meshes, $\nu=1$}
	\label{tab:errors1_standard}
	\begin{footnotesize}
		\begin{tabular}{@{}rllllllll@{}}
			\toprule
			\textbf{CR\hphantom{-BDM}}	& \multicolumn{4}{c}{$\mu = 1$} & \multicolumn{4}{c}{$\mu = 0.4$} \\ \cmidrule(lr){2-5} \cmidrule(l){6-9}
			ndof 	& $\norm{\vec{u}-\vec{u}_h}_{1,h}$	& eoc	& $\norm{p-p_h}_0$	& eoc & $\norm{\vec{u}-\vec{u}_h}_{1,h}$	& eoc	& $\norm{p-p_h}_0$	& eoc\\ \cmidrule(r){1-1} \cmidrule(lr){2-5} \cmidrule(l){6-9}
			894 &	6.8435e--01 &	 &	7.5632e--01 &	 &	6.9908e--01 &	 &	7.1907e--01 &	\\
			4137 &	5.1119e--01 &	0.68 &	5.8913e--01 &	0.58 &	4.8222e--01 &	0.73 &	4.3157e--01 &	1.00\\
			25650 &	3.5956e--01 &	0.58 &	3.6946e--01 &	0.77 &	2.9154e--01 &	0.83 &	2.1233e--01 &	1.17\\
			155364 &	2.3669e--01 &	0.69 &	2.0833e--01 &	0.95 &	1.6660e--01 &	0.93 &	9.9137e--02 &	1.27\\
			1376733 &	1.6167e--01 &	0.52 &	1.3838e--01 &	0.56 &	8.2279e--02 &	0.97 &	4.4658e--02 &	1.10\\ \bottomrule
		\end{tabular}
	\end{footnotesize}
	
	\vspace{1ex}
	\label{tab:errors1_rt}
	\begin{footnotesize}
		\begin{tabular}{@{}rllllllll@{}}
			\toprule
			\textbf{CR-RT}\hspace{1.2em}	& \multicolumn{4}{c}{$\mu = 1$} & \multicolumn{4}{c}{$\mu = 0.4$} \\ \cmidrule(lr){2-5} \cmidrule(l){6-9}
			ndof 	& $\norm{\vec{u}-\vec{u}_h}_{1,h}$	& eoc	& $\norm{p-p_h}_0$	& eoc & $\norm{\vec{u}-\vec{u}_h}_{1,h}$	& eoc	& $\norm{p-p_h}_0$	& eoc\\ \cmidrule(r){1-1} \cmidrule(lr){2-5} \cmidrule(l){6-9}
			894 &	6.6119e--01 &	 &	6.9348e--01 &	 &	6.6847e--01 &	 &	6.8824e--01 &	\\
			4137 &	4.9119e--01 &	0.69 &	5.4644e--01 &	0.56 &	4.5698e--01 &	0.74 &	4.0627e--01 &	1.03\\
			25650 &	3.4653e--01 &	0.57 &	3.5565e--01 &	0.71 &	2.7329e--01 &	0.85 &	1.9630e--01 &	1.20\\
			155364 &	2.3041e--01 &	0.68 &	2.0405e--01 &	0.92 &	1.5591e--01 &	0.93 &	9.3436e--02 &	1.24\\
			1376733 &	1.5921e--01 &	0.51 &	1.3745e--01 &	0.54 &	7.6745e--02 &	0.97 &	4.3269e--02 &	1.06\\ \bottomrule
		\end{tabular}
	\end{footnotesize}
	
	\vspace{1ex}
	\label{tab:errors1_bdm}
	\begin{footnotesize}
		\begin{tabular}{@{}rllllllll@{}}
			\toprule
			\textbf{CR-BDM}	& \multicolumn{4}{c}{$\mu = 1$} & \multicolumn{4}{c}{$\mu = 0.4$} \\ \cmidrule(lr){2-5} \cmidrule(l){6-9}
			ndof 	& $\norm{\vec{u}-\vec{u}_h}_{1,h}$	& eoc	& $\norm{p-p_h}_0$	& eoc & $\norm{\vec{u}-\vec{u}_h}_{1,h}$	& eoc	& $\norm{p-p_h}_0$	& eoc\\ \cmidrule(r){1-1} \cmidrule(lr){2-5} \cmidrule(l){6-9}
			894 &	6.6057e--01 &	 &	6.8820e--01 &	 &	6.6909e--01 &	 &	7.0451e--01 &	\\
			4137 &	4.9002e--01 &	0.70 &	5.4862e--01 &	0.53 &	4.5764e--01 &	0.74 &	4.1090e--01 &	1.06\\
			25650 &	3.4568e--01 &	0.57 &	3.5646e--01 &	0.71 &	2.7437e--01 &	0.84 &	1.9731e--01 &	1.21\\
			155364 &	2.3000e--01 &	0.68 &	2.0450e--01 &	0.92 &	1.5679e--01 &	0.93 &	9.3566e--02 &	1.24\\
			1376733 &	1.5905e--01 &	0.50 &	1.3758e--01 &	0.54 &	7.7344e--02 &	0.97 &	4.3282e--02 &	1.06\\ \bottomrule
		\end{tabular}
	\end{footnotesize}
\end{table}

\begin{table}[t]
	\centering
	\caption{Errors and experimental convergence orders of the standard and modified Crouzeix--Raviart methods on uniform and graded meshes, $\nu=10^{-1}$}
	\label{tab:errors2_standard}
	\begin{footnotesize}
		\begin{tabular}{@{}rllllllll@{}}
			\toprule
			\textbf{CR}\hspace{2.65em}	& \multicolumn{4}{c}{$\mu = 1$} & \multicolumn{4}{c}{$\mu = 0.4$} \\ \cmidrule(lr){2-5} \cmidrule(l){6-9}
			ndof 	& $\norm{\vec{u}-\vec{u}_h}_{1,h}$	& eoc	& $\norm{p-p_h}_0$	& eoc & $\norm{\vec{u}-\vec{u}_h}_{1,h}$	& eoc	& $\norm{p-p_h}_0$	& eoc\\ \cmidrule(r){1-1} \cmidrule(lr){2-5} \cmidrule(l){6-9}
			894 &	3.2893e+00 &	 &	4.7738e--01 &	 &	3.2855e+00 &	 &	4.6523e--01 &	\\
			4137 &	2.6582e+00 &	0.50 &	3.7183e--01 &	0.58 &	2.6933e+00 &	0.39 &	3.2386e--01 &	0.71\\
			25650 &	1.8748e+00 &	0.57 &	2.4716e--01 &	0.67 &	1.8576e+00 &	0.61 &	1.7577e--01 &	1.00\\
			155364 &	1.3170e+00 &	0.59 &	1.4109e--01 &	0.93 &	1.4685e+00 &	0.39 &	8.6913e--02 &	1.17\\
			1376733 &	8.5362e--01 &	0.59 &	9.9609e--02 &	0.48 &	1.6304e+00 &	--0.14 &	3.9505e--02 &	1.08\\ \bottomrule
		\end{tabular}
	\end{footnotesize}
	
	\vspace{1ex}
	\label{tab:errors2_rt}
	\begin{footnotesize}
		\begin{tabular}{@{}rllllllll@{}}
			\toprule
			\textbf{CR-RT}\hspace{1.2em}	& \multicolumn{4}{c}{$\mu = 1$} & \multicolumn{4}{c}{$\mu = 0.4$} \\ \cmidrule(lr){2-5} \cmidrule(l){6-9}
			ndof 	& $\norm{\vec{u}-\vec{u}_h}_{1,h}$	& eoc	& $\norm{p-p_h}_0$	& eoc & $\norm{\vec{u}-\vec{u}_h}_{1,h}$	& eoc	& $\norm{p-p_h}_0$	& eoc\\ \cmidrule(r){1-1} \cmidrule(lr){2-5} \cmidrule(l){6-9}
			894 &	6.6352e--01 &	 &	3.7503e--01 &	 &	6.6805e--01 &	 &	3.7467e--01 &	\\
			4137 &	4.9366e--01 &	0.69 &	2.8017e--01 &	0.68 &	4.5699e--01 &	0.74 &	2.4133e--01 &	0.86\\
			25650 &	3.5092e--01 &	0.56 &	1.9315e--01 &	0.61 &	2.7355e--01 &	0.84 &	1.3218e--01 &	0.99\\
			155364 &	2.3511e--01 &	0.67 &	1.2179e--01 &	0.77 &	1.5605e--01 &	0.93 &	7.0937e--02 &	1.04\\
			1376733 &	1.6308e--01 &	0.50 &	8.3975e--02 &	0.51 &	7.6794e--02 &	0.97 &	3.4700e--02 &	0.98\\ \bottomrule
		\end{tabular}
	\end{footnotesize}
	
	\vspace{1ex}
	\label{tab:errors2_bdm}
	\begin{footnotesize}
		\begin{tabular}{@{}rllllllll@{}}
			\toprule
			\textbf{CR-BDM}	& \multicolumn{4}{c}{$\mu = 1$} & \multicolumn{4}{c}{$\mu = 0.4$} \\ \cmidrule(lr){2-5} \cmidrule(l){6-9}
			ndof 	& $\norm{\vec{u}-\vec{u}_h}_{1,h}$	& eoc	& $\norm{p-p_h}_0$	& eoc & $\norm{\vec{u}-\vec{u}_h}_{1,h}$	& eoc	& $\norm{p-p_h}_0$	& eoc\\ \cmidrule(r){1-1} \cmidrule(lr){2-5} \cmidrule(l){6-9}
			894 &	6.6720e--01 &	 &	3.7487e--01 &	 &	6.6963e--01 &	 &	3.7481e--01 &	\\
			4137 &	4.9503e--01 &	0.70 &	2.8009e--01 &	0.68 &	4.5775e--01 &	0.74 &	2.4137e--01 &	0.86\\
			25650 &	3.5253e--01 &	0.56 &	1.9308e--01 &	0.61 &	2.7426e--01 &	0.84 &	1.3218e--01 &	0.99\\
			155364 &	2.3649e--01 &	0.66 &	1.2175e--01 &	0.77 &	1.5646e--01 &	0.93 &	7.0934e--02 &	1.04\\
			1376733 &	1.6359e--01 &	0.50 &	8.3949e--02 &	0.51 &	7.8260e--02 &	0.95 &	3.4699e--02 &	0.98\\ \bottomrule
		\end{tabular}
	\end{footnotesize}
\end{table}

The calculations were performed with parameter values $\nu \in \{10^{-1},1\}$ and $\mu \in \{0.4, 1\}$. 
Tables \ref{tab:errors1_standard} and \ref{tab:errors2_standard} contain the computed errors. Comparing the estimated order of convergence ($eoc$) for meshes without grading, $\mu=1$, and with grading towards the edge, $\mu=0.4$, shows that anisotropic grading recovers the optimal convergence rate for all methods. 
The results with viscosity $\nu = 10^{-1}$ show the pressure-robustness of the modified method, as the absolute value of the velocity error does not depend on $\nu$, contrary to the standard method. 
The modified method seems to perform optimally in the anisotropic setting even with the low regularity data in the case $\nu\neq 1$, where the optimal convergence rate could not be observed with the standard method.

\begin{remark}
	The data function \eqref{eq:datafunction} is not in $\vec{L}^2(\Omega)$ for $\nu \neq 1$, but the right hand side integrals for our methods are still finite. 
	However, in order to produce the shown results in \Cref{tab:errors2_standard} the numerical quadrature for the linear form had to be highly accurate.
	For our CR-BDM calculations, additionally to choosing a high quadrature degree as for the other methods, we used local mesh refinement near the singular axis.
\end{remark}
\begin{remark}
	The quadrature procedure described in the last remark did not improve the convergence results of the standard method on graded meshes with $\nu\neq 1$, where the optimal rate could not be observed. Although we do not have a detailed proof, this seems to be a result of $\vec{f} \notin \vec{L}^2(\Omega)$:
	
	The velocity error estimate from \cite{ApelNicaiseSchoberl2001} for the standard method, which is shown for $\vec{f}\in \vec{L}^2(\Omega)$, comprises the consistency error and the best approximation error, the latter being bounded in terms of the interpolation error of the Crouzeix--Raviart interpolation. While we could see the interpolation error in this test converging optimally on the graded meshes, the consistency error does not seem to converge for irregular data. In contrast to the standard Crouzeix--Raviart method, the proof of our pressure robust estimate from \Cref{sec:error} only needs to bound the consistency error for the Helmholtz--Hodge projection of the data, which, for this example, is in $\vec{L}^2(\Omega)$. This is the reason why the modified methods work.
	
	Since the consistency error estimate from \cite{ApelNicaiseSchoberl2001} was prepared in \cite{ApelNicaiseSchoberl2001:2} with a similar estimate for the Poisson equation, we did a further test computation for the Poisson problem with exact solution $u = r^{\nicefrac{1}{2}} \sin(\nicefrac{2}{3}\varphi)$ on the same meshes. For this exact solution the data is not in $L^2(\Omega)$ as in the Stokes case, and the results showed a similar convergence behavior as the Stokes example. This is another indication that the consistency error of the Crouzeix--Raviart method causes the bad numerical performance.
\end{remark}

\subsection{Example with fixed data}
Consider the same general setting as in the previous example. We now use the data
\begin{equation*}
	\vec{f} = 	\vec{f}_0 + \nabla \phi_i,\quad i\in\{1,2\},
\end{equation*}
where $\vec{f}_0$ and $\phi_i$ are chosen as 
\begin{align*}
	&\vec{f}_0 = \begin{pmatrix}
					0\\
					0\\
					r^{\lambda-1} \Phi(\varphi)
				\end{pmatrix}, 
	&&\phi_1 = 0, &&\phi_2 = 10 r^{\lambda} \Phi(\varphi),
\end{align*}
with $\Phi(\varphi)$ from \eqref{eq:Phi}.
The function $\vec{f}_0$ is aquired by setting $\nu = 1$ in \eqref{eq:datafunction} and the functions $\phi_i$ are used to show the pressure-robustness in the case of a scaled exact velocity solution: the errors for the CR-RT and CR-BDM methods do not change when adding gradient fields to the data. The exact solutions for the convergence analysis can be deduced from the first example using the considerations from \Cref{lem:solution_change} and \Cref{rem:data_exact_solution}.

As before we have $-\Delta\vec{u}\notin \vec{L}^2(\Omega)$, but due to our choice of the functions $\vec{f}_0$ and $\phi_i$ we now get $\vec{f}\in \vec{L}^2(\Omega)$ for all calculations. 

The calculations were performed with viscosity parameter $\nu \in \{10^{-3},1\}$ and, since the difference in convergence orders was already demonstrated in the previous example, we only use anisotropic meshes with grading parameter $\mu = 0.4$. 

\begin{table}[t]
	\centering
	\caption{Errors and experimental convergence orders of the standard and modified Crouzeix--Raviart methods on graded meshes with $\mu = 0.4$, $\nu=1$}
	\label{tab:errors_standard_alt}
	\begin{footnotesize}
		\begin{tabular}{@{}rllllllll@{}}
			\toprule
			\textbf{CR}\hphantom{\textbf{-BDM}}	& \multicolumn{4}{c}{$\phi_1$} & \multicolumn{4}{c}{$\phi_2$} \\ \cmidrule(lr){2-5} \cmidrule(l){6-9}
			ndof 	& $\norm{\vec{u}-\vec{u}_h}_{1,h}$	& eoc	& $\norm{p-p_h}_0$	& eoc & $\norm{\vec{u}-\vec{u}_h}_{1,h}$	& eoc	& $\norm{p-p_h}_0$	& eoc\\ \cmidrule(r){1-1} \cmidrule(lr){2-5} \cmidrule(l){6-9}
			894 &	6.9908e--01 &	 &	7.1907e--01 &	 &	1.6311e+00 &	 &	4.7159e+00 &	\\
			4137 &	4.8222e--01 &	0.73 &	4.3157e--01 &	1.00 &	1.2720e+00 &	0.49 &	2.8249e+00 &	1.00\\
			25650 &	2.9154e--01 &	0.83 &	2.1233e--01 &	1.17 &	8.3565e--01 &	0.69 &	1.2864e+00 &	1.29\\
			155364 &	1.6660e--01 &	0.93 &	9.9137e--02 &	1.27 &	5.3725e--01 &	0.74 &	5.8967e--01 &	1.30\\
			1376733 &	8.2279e--02 &	0.97 &	4.4658e--02 &	1.10 &	2.6864e--01 &	0.95 &	2.1364e--01 &	1.40\\ \bottomrule
		\end{tabular}
	\end{footnotesize}
	
	\vspace{1ex}
	\label{tab:errors_rt_alt}
	\begin{footnotesize}
		\begin{tabular}{@{}rllllllll@{}}
			\toprule
			\textbf{CR-RT}\hspace{1.2em}	& \multicolumn{4}{c}{$\phi_1$} & \multicolumn{4}{c}{$\phi_2$} \\ \cmidrule(lr){2-5} \cmidrule(l){6-9}
			ndof 	& $\norm{\vec{u}-\vec{u}_h}_{1,h}$	& eoc	& $\norm{p-p_h}_0$	& eoc & $\norm{\vec{u}-\vec{u}_h}_{1,h}$	& eoc	& $\norm{p-p_h}_0$	& eoc\\ \cmidrule(r){1-1} \cmidrule(lr){2-5} \cmidrule(l){6-9}
			894 &	6.6847e--01 &	 &	6.8824e--01 &	 &	6.6846e--01 &	 &	2.0599e+00 &	\\
			4137 &	4.5698e--01 &	0.74 &	4.0628e--01 &	1.03 &	4.5698e--01 &	0.74 &	1.2274e+00 &	1.01\\
			25650 &	2.7329e--01 &	0.75 &	1.9630e--01 &	1.20 &	2.7329e--01 &	0.85 &	6.5688e--01 &	1.03\\
			155364 &	1.5591e--01 &	0.93 &	9.3437e--02 &	1.24 &	1.5591e--01 &	0.93 &	3.6601e--01 &	0.97\\
			1376733 &	7.6745e--02 &	0.97 &	4.3269e--02 &	1.06 &	7.6745e--02 &	0.97 &	1.6909e--01 &	1.06\\ \bottomrule
		\end{tabular}
	\end{footnotesize}
	
	\vspace{1ex}
	\label{tab:errors_bdm_alt}
	\begin{footnotesize}
		\begin{tabular}{@{}rllllllll@{}}
			\toprule
			\textbf{CR-BDM}	& \multicolumn{4}{c}{$\phi_1$} & \multicolumn{4}{c}{$\phi_2$} \\ \cmidrule(lr){2-5} \cmidrule(l){6-9}
			ndof 	& $\norm{\vec{u}-\vec{u}_h}_{1,h}$	& eoc	& $\norm{p-p_h}_0$	& eoc & $\norm{\vec{u}-\vec{u}_h}_{1,h}$	& eoc	& $\norm{p-p_h}_0$	& eoc\\ \cmidrule(r){1-1} \cmidrule(lr){2-5} \cmidrule(l){6-9}
			894 &	6.6909e--01 &	 &	7.0451e--01 &	 &	6.6909e--01 &	 &	2.0654e+00 &	\\
			4137 &	4.5764e--01 &	0.74 &	4.1090e--01 &	1.06 &	4.5763e--01 &	0.74 &	1.2290e+00 &	1.02\\
			25650 &	2.7437e--01 &	0.84 &	1.9731e--01 &	1.21 &	2.7437e--01 &	0.84 &	6.5718e--01 &	1.03\\
			155364 &	1.5679e--01 &	0.93 &	9.3566e--02 &	1.24 &	1.5679e--01 &	0.93 &	3.6604e--01 &	0.97\\
			1376733 &	7.7344e--02 &	0.97 &	4.3282e--02 &	1.06 &	7.7344e--02 &	0.97 &	1.6909e--01 &	1.06\\ \bottomrule
		\end{tabular}
	\end{footnotesize}
\end{table}

\begin{table}[t]
	\centering
	\caption{Errors and experimental convergence orders of the standard and modified Crouzeix--Raviart methods on graded meshes with $\mu=0.4$, $\nu=10^{-3}$}
	\label{tab:errors2_standard_alt}
	\begin{footnotesize}
		\begin{tabular}{@{}rllllllll@{}}
			\toprule
			\textbf{CR}\hphantom{\textbf{-BDM}}	& \multicolumn{4}{c}{$\phi_1$} & \multicolumn{4}{c}{$\phi_2$} \\ \cmidrule(lr){2-5} \cmidrule(l){6-9}
			ndof 	& $\norm{\vec{u}-\vec{u}_h}_{1,h}$	& eoc	& $\norm{p-p_h}_0$	& eoc & $\norm{\vec{u}-\vec{u}_h}_{1,h}$	& eoc	& $\norm{p-p_h}_0$	& eoc\\ \cmidrule(r){1-1} \cmidrule(lr){2-5} \cmidrule(l){6-9}
			894 &	6.9908e+02 &	 &	7.1906e--01 &	 &	1.6311e+03 &	 &	4.7159e+00 &	\\
			4137 &	4.8222e+02 &	0.73 &	4.3157e--01 &	1.00 &	1.2720e+03 &	0.49 &	2.8249e+00 &	1.00\\
			25650 &	2.9154e+02 &	0.83 &	2.1233e--01 &	1.17 &	8.3565e+02 &	0.69 &	1.2864e+00 &	1.29\\
			155364 &	1.6660e+02 &	0.93 &	9.9137e--02 &	1.27 &	5.3725e+02 &	0.74 &	5.8967e--01 &	1.30\\
			1376733 &	8.2279e+01 &	0.97 &	4.4658e--02 &	1.10 &	2.6864e+02 &	0.95 &	2.1364e--01 &	1.40\\ \bottomrule
		\end{tabular}
	\end{footnotesize}
	
	\vspace{1ex}
	\label{tab:errors2_rt_alt}
	\begin{footnotesize}
		\begin{tabular}{@{}rllllllll@{}}
			\toprule
			\textbf{CR-RT}\hspace{1.2em}	& \multicolumn{4}{c}{$\phi_1$} & \multicolumn{4}{c}{$\phi_2$} \\ \cmidrule(lr){2-5} \cmidrule(l){6-9}
			ndof 	& $\norm{\vec{u}-\vec{u}_h}_{1,h}$	& eoc	& $\norm{p-p_h}_0$	& eoc & $\norm{\vec{u}-\vec{u}_h}_{1,h}$	& eoc	& $\norm{p-p_h}_0$	& eoc\\ \cmidrule(r){1-1} \cmidrule(lr){2-5} \cmidrule(l){6-9}
			894 &	6.6847e+02 &	 &	6.8825e--01 &	 &	6.6846e+02 &	 &	2.0599e+00 &	\\
			4137 &	4.5698e+02 &	0.74 &	4.0628e--01 &	1.03 &	4.5698e+02 &	0.74 &	1.2274e+00 &	1.01\\
			25650 &	2.7329e+02 &	0.85 &	1.9630e--01 &	1.20 &	2.7329e+02 &	0.85 &	6.5688e--01 &	1.03\\
			155364 &	1.5591e+02 &	0.93 &	9.3437e--02 &	1.24 &	1.5591e+02 &	0.93 &	3.6601e--01 &	0.97\\
			1376733 &	7.6745e+01 &	0.97 &	4.3269e--02 &	1.06 &	7.6745e+01 &	0.97 &	1.6909e--01 &	1.06\\ \bottomrule
		\end{tabular}
	\end{footnotesize}
	
	\vspace{1ex}
	\label{tab:errors2_bdm_alt}
	\begin{footnotesize}
		\begin{tabular}{@{}rllllllll@{}}
			\toprule
			\textbf{CR-BDM}	& \multicolumn{4}{c}{$\phi_1$} & \multicolumn{4}{c}{$\phi_2$} \\ \cmidrule(lr){2-5} \cmidrule(l){6-9}
			ndof 	& $\norm{\vec{u}-\vec{u}_h}_{1,h}$	& eoc	& $\norm{p-p_h}_0$	& eoc & $\norm{\vec{u}-\vec{u}_h}_{1,h}$	& eoc	& $\norm{p-p_h}_0$	& eoc\\ \cmidrule(r){1-1} \cmidrule(lr){2-5} \cmidrule(l){6-9}
			894 &	6.6909e+02 &	 &	7.0451e--01 &	 &	6.6909e+02 &	 &	2.0654e+00 &	\\
			4137 &	4.5764e+02 &	0.74 &	4.1090e--01 &	1.06 &	4.5763e+02 &	0.74 &	1.2290e+00 &	1.02\\
			25650 &	2.7437e+02 &	0.84 &	1.9731e--01 &	1.21 &	2.7437e+02 &	0.84 &	6.5718e--01 &	1.03\\
			155364 &	1.5679e+02 &	0.93 &	9.3566e--02 &	1.24 &	1.5679e+02 &	0.93 &	3.6604e--01 &	0.97\\
			1376733 &	7.7344e+01 &	0.97 &	4.3282e--02 &	1.06 &	7.7344e+01 &	0.97 &	1.6909e--01 &	1.06\\ \bottomrule
		\end{tabular}
	\end{footnotesize}
\end{table}
Tables \ref{tab:errors_standard_alt} and \ref{tab:errors2_standard_alt} show the errors for both choices of $\phi_i$. 
We see that while the asymptotic convergence rates are optimal for all methods, the additional gradient part $\nabla \phi_2$ in the data has a significant influence on the value of the velocity error of the standard method. 
In contrast, the modified methods show their pressure-robustness by yielding the same velocity solution, and thus unchanged velocity errors. The scaling of the velocity solution with $\nu^{-1}$ for fixed $\vec{f}$ is clearly visible when comparing the two tables.

\printbibliography

\end{document}